\documentclass[10pt]{article}

\usepackage{amsmath}
\usepackage{amssymb}
\usepackage{amsthm}
\usepackage{graphicx,color}
\usepackage{texdraw}
\newtheoremstyle{plain2}{\topsep}{\topsep}%
     {\itshape}
     {}
     {\bfseries}
     {.}
     {.5em}
     {\thmnumber{(#2)}\thmname{ #1}\thmnote{ #3}}

\theoremstyle{plain2}
\newtheorem{teo}{Theorem}[section]
\newtheorem{prop}[teo]{Proposition}
\newtheorem{coro}[teo]{Corollary}
\newtheorem{lemma}[teo]{Lemma}

\newtheoremstyle{definition2}{\topsep}{\topsep}%
     {}
     {}
     {\bfseries}
     {.}
     {.5em}
     {\thmnumber{(#2)}\thmname{ #1}\thmnote{ #3}}

\theoremstyle{definition2}

\newcommand{\sezione}[1]{\section{#1}\setcounter{equation}{0}}

\def\R{\mathbb{R}}

\def\Om{\Omega}
\def\vphi{\varphi}

\begin{document}

\title{\sc Positive constrained minimizers for supercritical problems in the ball}

\author{ Massimo
Grossi\thanks{Dipartimento di Matematica, Universit\`a di Roma
``La Sapienza", P.le A. Moro 2 - 00185 Roma, e-mail: {\sf
grossi@mat.uniroma1.it}} \and Benedetta Noris\thanks{Dipartimento di Matematica e Applicazioni, Universit\`a degli Studi di Milano-Bicocca, Via Bicocca degli Arcimboldi 8 - 20126 Milano, e-mail: {\sf benedettanoris@gmail.com}}}
\date{}
\maketitle

\begin{abstract}
We provide a sufficient condition for the existence of a positive solution to
$$
-\Delta u+V(|x|) u=u^p \quad\hbox{ in } B_1,
$$
when $p$ is large enough. Here $B_1$ is the unit ball of $\R^n$, $n\ge2$, and we deal both with Neumann and Dirichlet homogeneous boundary conditions.
The solution turns to be a constrained minimum of the associated energy functional.
As an application we show that, in case $V(|x|)\geq 0$, $V\not\equiv0$ is smooth and $p$ is sufficiently large, the Neumann problem always admits a solution.
\end{abstract}

\sezione{Introduction}
In this paper we study the existence of radial solutions to the following equation,
\begin{equation}\label{eq:main_nonlinear_equation}
\left\{ \begin{array}{ll}
-\Delta u+V(|x|) u=u^p \quad &\text{ in } B_1 \\
u>0 \quad &\text{ in } B_1,
\end{array}\right.
\end{equation}
both with Neumann and Dirichlet homogeneous boundary conditions. Here $B_1$ is the unit ball of $\R^n$, $n\geq2$, and $V(|x|)\geq0, V\not\equiv0$ is a smooth, radial function. We are interested when the exponent $p$ is large. Recent results (see \cite{G2,C}) suggest that the existence of solutions of \eqref{eq:main_nonlinear_equation} is related to the critical points of a function $F(r)$, associated to this equation in the limit as $p\to+\infty$ (see \eqref{eq:definition_F} below). Our aim is to extend the known existence results in this direction through a better understanding of $F(r)$.

In order to be more precise, let us start by considering Neumann boundary conditions. We denote by $G(r,s)$ the Green function of the operator
\begin{equation}\label{eq:linear_radial_operator}
{\cal L}u=-u''-\frac{n-1}ru'+V(r)u, \qquad u'(0)=0
\end{equation}
with the Neumann boundary condition $u'(1)=0$. Note that, unlike the case of higher dimensions, the Green function is bounded, hence $G(r,r)$ makes sense and we can define
\begin{equation}\label{eq:definition_F}
F(r)= \frac{|\partial B_1| r^{n-1}}{G(r,r)}, \quad r\in(0,1].
\end{equation}
Here $|\partial B_1|$ is the measure of the boundary of the unit sphere.
$F(r)$ represents the energy naturally associated to the normalized Green function $G(\cdot,r)/G(r,r)$ (see Lemma \ref{lemma:relation_F_Q}). Our first result concerns the existence of solutions for the Neumann problem.
\begin{teo}\label{teo:main_theorem}
Let us consider the problem
\begin{equation}\label{eq:main_nonlinear_equation_neumann}
\left\{ \begin{array}{ll}
-\Delta u+V(|x|) u=u^p \quad &\text{ in } B_1 \\
u>0 \quad &\text{ in } B_1,\\
\frac{\partial u}{\partial\nu}=0&\text{ on }\partial B_1
\end{array}\right.
\end{equation}
where $V(|x|)\geq0, V\not\equiv0$ is a smooth, radial function in $B_1$, such that $F(r)$ admits a local minimum point at $\overline r\in(0,1]$. Then, for $p$ large enough, there exists a radial solution $u_p=u_p(r)$ to \eqref{eq:main_nonlinear_equation_neumann} which verifies
\begin{equation}\label{5}
u_p(r)\to \frac{G(r,\overline r)}{G(\overline r,\overline r)} \quad \text{ as } p\to+\infty \qquad \text{ in } H^1(B_1)\cap L^\infty(B_1).
\end{equation}
\end{teo}
We point out that we find as many different solutions as the number of local minimum points of $F(r)$ in $(0,1]$.\\
Note that if $\bar r=1$, we define $G(r,1)$ as the punctual limit of $G(r,s)$ as $s\to 1$ (see Section \ref{sec:proof_teo_Neumann_bar_r=1}).
Moreover, being $\bar r=1$ always a local minimum point of $F(r)$ (see Section \ref{sec:proof_teo_Neumann_bar_r=1}) we deduce the following result,
\begin{teo}\label{teo:Neumann_bar_r=1}
Let $V(|x|)\geq0, V\not\equiv0$ be a smooth, radial function in $B_1$.
Then, for $p$ large enough, there exists a radial solution $u_p=u_p(r)$ to \eqref{eq:main_nonlinear_equation_neumann} which converges to $G(r,1)/G(1,1)$.
\end{teo}
In particular, from \eqref{5}, we derive the following new existence result for the problem with constant potential.
\begin{coro}\label{i2}
Let $V(|x|)\equiv\lambda>0$.
Then, for $p$ large enough, there exists a nonconstant radial solution $u_p=u_p(r)$ to \eqref{eq:main_nonlinear_equation_neumann}.
\end{coro}

These results continue the study of the supercritical case started in \cite{G1} and \cite{G2}, with zero Dirichlet boundary conditions and $p$ large. However, here we have some important news. The first one concerns the technique used in the proof of Theorem \ref{teo:main_theorem}. In \cite{G2}, a crucial point in the construction of the solution was given by the following "limit problem",
\begin{equation}\label{6}
-U''=e^U\quad\hbox{in }\R
\end{equation}
and by the corresponding linearized equation. The solution was then found "close" to a projection of a suitable solution to \eqref{6}. This approach is quite standard in this type of problems (there is a very wide literature on the topic), but involves heavy calculations. In the proof of Theorem \ref{teo:main_theorem} we do not use the limit problem (\ref{6}) but we find the solution using some suitable constrained variational approach. Note that a similar idea was used in \cite{P} to handle the supercritical problem. This technique, in the opinion of the authors, makes the proofs much simpler. Moreover, we think that similar ideas could be used in analogous problems with lack of compactness. Another important advantage resulting from this technique is that it does not require any non-degeneracy assumption on the minimum point $\overline r$, hypothesis which is not easy to verify.

As it concerns Dirichlet boundary conditions, we partially recover, through this different technique, the results in \cite{G2}. In fact the analogous of Theorem \ref{teo:main_theorem} holds, in the following form.
\begin{teo}\label{i3}
Let us consider the problem
\begin{equation}\label{i4}
\left\{ \begin{array}{ll}
-\Delta u+V(|x|) u=u^p \quad &\text{ in } B_1 \\
u>0 \quad &\text{ in } B_1,\\
u=0&\text{ on }\partial B_1
\end{array}\right.
\end{equation}
where $V(|x|)\geq0, V\not\equiv0$ is a smooth, radial function in $B_1$, such that $F(r)$ admits a local minimum point at $\overline r\in(0,1)$. Then, for $p$ large enough, there exists a radial solution $u_p=u_p(r)$ to \eqref{i4} which verifies
\begin{equation}\label{i5}
u_p(r)\to \frac{G(r,\overline r)}{G(\overline r,\overline r)} \quad \text{ as } p\to+\infty \qquad \text{ in } H^1(B_1)\cap L^\infty(B_1).
\end{equation}
\end{teo}
Of course, in this case, the term $G(r,s)$ appearing in \eqref{eq:definition_F} and \eqref{i5} is the Green function of the operator \eqref{eq:linear_radial_operator} with the Dirichlet boundary condition $u(1)=0$. The proof of Theorem \ref{i3} is the same as the one of Theorem \ref{teo:main_theorem} (it is even easier because we do not need to analyze the case $\bar r=1$), for this reason we omit it. Actually, our technique gives a unified proof for both Dirichlet and Neumann boundary conditions. Moreover, again in the Dirichlet case, Catrina proved in \cite{C} that the condition in Theorem \ref{i3} is "almost" necessary. Indeed he proved the following result.
\begin{teo}\label{i6}
If the function $F_p(r)=\frac{r^{\frac{p-1}{p+3}(n-1)}}{G(r,r)}$ is monotonic, nonconstant, then
problem \eqref{i4} has no solution.
\end{teo}
Since $F_p\rightarrow F$ uniformly in any interval $[r_0,1]$ as $p\rightarrow+\infty$, we have that the existence of a minimum to the function $F$ becomes "almost" necessary for the existence of a solution.

We end this section with a brief history of the problem \eqref{eq:main_nonlinear_equation}. First, if $1<p<\frac{n+2}{n-2}$ for $n\ge3$ ({\em subcritical case}) and $p>1$ if $n=2$, it is not difficult to prove the existence of a solution. This can be shown observing that the following infimum
\begin{equation}\label{2}
S_p=\inf\left\{\int_\Om\left(|\nabla u|^2+V(x) u^2\right):\ \int_\Om|u|^{p+1}=1,\ u\in H^1(\Om)
\right\},
\end{equation}
is achieved because the compact embedding of $H^1(\Omega)$ in $L^{p+1}(\Omega)$ (the same holds for $H^1_0(\Omega)$).

If $p=\frac{n+2}{n-2}$ for $n\ge3$ ({\em critical case}), it is well known that the existence of a solutions to \eqref{eq:main_nonlinear_equation} depends on the shape of $\Om$ and on the properties of $V(x)$. Since there is a huge litarature on this topic we just mention the pioneering papers by Brezis and Nirenberg \cite{BN} and Bahri and Coron \cite{BC} for the Dirichlet case.
The Neumann  problem \eqref{eq:main_nonlinear_equation} was first studied when $V(x)$ is a positive constant $\lambda$ and some existence results were established in \cite{N, CK91,AY91,AY93,NPT,APY}. See also \cite{AM,AY} for the case where $V$ is not constant.

The {\em supercritical case} $p>\frac{n+2}{n-2}$ is much more difficult to handle since there is no embedding of $H^1(\Omega)$ in $L^{p+1}(\Omega)$.
A consequence of this fact is that the infimum $S_p$ in \eqref{2} is $zero$ and hence it can not be used to find a solution to \eqref{eq:main_nonlinear_equation}.
Some interesting existence and nonexistence results in special domains with Dirichlet boundary conditions are due to Passaseo (\cite{P,P1}). We also mention the recent paper \cite{DW}, concerning domains with small circular holes.
We emphasize that the case of a general domain seems not yet fully understood. Regarding the case of the ball, to our knowledge the only results are those of \cite{G2} (already mentioned above) and \cite{MP,MPS} (here $V$ is constant).

Unlike the Dirichlet case, where in recent years there have been several developments, in the Neumann case there is a very poor literature.
To our knowledge the only results in the supercritical case are due to Ni (see \cite{N}) and Lin-Ni (see \cite{LN}). In particular the authors prove the following.
\begin{teo}\label{i7}
Let us consider the problem,
\begin{equation}\label{i8}
\left\{ \begin{array}{ll}
-\Delta u+\lambda u=u^p \quad &\text{ in } B_1 \\
u>0 \quad &\text{ in } B_1,\\
\frac{\partial u}{\partial\nu}=0&\text{ on }\partial B_1
\end{array}\right.
\end{equation}
Then, there exist positive constants $\lambda_0=\lambda_0(n,p)$ and $\lambda_1=\lambda_1(n)$ such that\\
i) for any $\lambda>\lambda_1$ there exists at least a nonconstant radial solution to \eqref{eq:main_nonlinear_equation_neumann},\\
ii) for any $\lambda<\lambda_0$ \eqref{eq:main_nonlinear_equation_neumann} does not admit any nonconstant radial solution.
\end{teo}
From this result and Corollary \eqref{i2}, we derive that the constant $\lambda_0(n,p)\rightarrow0$ as $p\rightarrow+\infty$.

The paper is organized as follows. Sections \ref{s0}-\ref{sec:proof_main_theorem} contain the proof of Theorem \ref{teo:main_theorem} in the case $\bar r \in (0,1)$.
In Section \ref{s0} we introduce a family of variational problems depending on a parameter $p\in (1,\infty)$ and a limit problem.
The existence and convergence of the minimizers is shown in Section \ref{sec:existence_convergence_minimizers}. In Section \ref{sec:proof_main_theorem} we end the proof of Theorem \ref{teo:main_theorem} and present an additional property of the solution (see Proposition \ref{prop:reflection_principle}).
Section \ref{sec:proof_teo_Neumann_bar_r=1} deals with the case $\bar r=1$ and with Theorem \ref{teo:Neumann_bar_r=1}.
Finally, in the Appendix we collect some properties of the Green function.

\sezione{Variational setting and notations}\label{s0}
We introduce the Sobolev space of radial functions
\begin{equation}
H^1_r(B_1)=\{ u\in H^1(B_1):\ u=u(|x|)\}.
\end{equation}
In the following we will often make the abuse of notation $u(r)=u(|x|)$. We find solutions to \eqref{eq:main_nonlinear_equation_neumann} as constrained minimizers (in this space) of the energy functional
\begin{equation}
Q(u)= \int_{B_1} \left[ |\nabla u|^2 + V(|x|) u^2 \right],
\end{equation}
under the standard $L^p$--mass constraint and under an additional constraint which will be proven to be natural for $p$ sufficiently large. Let $\bar r \in (0,1)$ be a local minimum point of $F(r)$, then there exist $0<R_1<R_2<1$ such that $\bar r$ is a global minimum point in $[R_1,R_2]$. Set
\begin{equation*}
K_p=\left\{u\in {H^1_r(B_1)}:\ \left(|B_1|^{-1}\int_{B_1} |u|^{p+1}\right)^{\frac{1}{p+1}}=1 \ \hbox{ and }\ |u|\leq c \text{ in } B_{R_1}\cup(B_1\setminus B_{R_2}) \right\}
\end{equation*}
where $B_R$ denotes the ball centered at the origin of radius $R$ and $c$ satisfies
\begin{equation}\label{eq:property_c}
\max\left\{ \frac{G(R_1,\bar r)}{G(\bar r,\bar r)}, \frac{G(R_2,\bar r)}{G(\bar r,\bar r)} \right\} <c<1.
\end{equation}
Notice that, by the strong maximum principle, for every $r\neq s$ it holds $G(r,s)<G(s,s)$, hence it is always possible to find $c$ as in \eqref{eq:property_c}.
We consider the following infimum
\begin{equation}\label{8}
J_p=\inf\left\{Q(u),\ u\in K_p\right\}
\end{equation}
Of course, a nonnegative function which achieves $J_p$ provides a solution (up to a multiplicative constant) to (\ref{eq:main_nonlinear_equation_neumann}) if $u<c$ in $B_{R_1}\cup(B_1\setminus B_{R_2})$. Actually, we will see that the condition $|u|\leq c$ in $B_{R_1}$ prevents the solutions from concentrating around the origin and $|u|\leq c$ in $B_1\setminus B_{R_2}$ forces the solution to have its maximum around the local minimum point $\bar r$.

Our strategy is based on the asymptotic analysis of the minimization problem $J_p$ as $p\to\infty$. In fact, we will show the convergence of $J_p$ to the limit following infimum,
\begin{equation}\label{9}
J_\infty=\inf\left\{Q(u),\ u\in K_\infty\right\}.
\end{equation}
where $K_\infty$ is
\begin{equation}\label{9a}
K_\infty=\left\{u\in {H^1_r(B_1)}: \ ||u||_\infty=1\hbox{ and }|u|\leq c \text{ in } B_{R_1}\cup(B_1\setminus B_{R_2}) \right\}.
\end{equation}
A key property is the following lemma, which enlightens our choice of the constant $c$ in \eqref{eq:property_c}.
\begin{lemma}\label{lemma:relation_F_Q}
For every $r\in (0,1)$ it holds
\begin{equation}
F(r)=Q\left(\frac{G(\cdot,r)}{G(r,r)}\right).
\end{equation}
\end{lemma}
\begin{proof}
By passing to polar coordinates we obtain
\begin{equation}\label{eq:relation_F_Q}
Q\left(\frac{G(\cdot,r)}{G(r,r)}\right)=
\frac{|\partial B_1|}{G(r,r)^2}\int_0^1\left[ \left(\frac{\partial G(t,r)}{\partial t}\right)^2+V(t)G(t,r)^2\right] t^{n-1}dt.
\end{equation}
On the other hand, by the definition of the Green function, we have
\begin{equation}
-\frac{\partial^2G(t,r)}{\partial t^2}-\frac{n-1}{t}\frac{\partial G(t,r)}{\partial t}+V(t)G(t,r)=\delta_r
\end{equation}
in the sense of distributions. Multiplying the last equation by $t^{n-1}G(t,r)$, integrating over $(0,1)$ and finally substituting in \eqref{eq:relation_F_Q}, we obtain the thesis.
\end{proof}

\sezione{Existence and convergence of the constrained minimizers}\label{sec:existence_convergence_minimizers}
Let us start by proving the existence of a minimizer to $J_p$.
\begin{prop}\label{prop:existence_u_p}
There exists a nonnegative function $u_p\in K_p$ such that $J_p=Q(u_p)$.
\end{prop}
\begin{proof}
Let $u_p^n\in K_p$ be a minimizing sequence for $J_p$, that is
\begin{equation}\label{a1}
\lim_{n\to\infty}Q(u_p^n)=J_p.
\end{equation}
Being $V(|x|)\geq0, V\not\equiv0$, we infer that $\{u_p^n\}_n$ is bounded in ${H^1_r(B_1)}$, hence it converges weakly (up to subsequences) to some function $u_p\in {H^1_r(B_1)}$ and almost everywhere in $B_1$. Now, it is immediate to check that $u_p\in K_p$, since
\begin{equation}\label{a2}
\|u_p^n-u_p\|_{L^p(B_1)}=\|u_p^n-u_p\|_{L^p(B_{R_1})}+\|u_p^n-u_p\|_{L^p(B_1\setminus B_{R_1})},
\end{equation}
and both terms tend to zero (as $n\to\infty$); the first one, since $|u_p^n|\le c$, by the Lebesgue convergence theorem and the second one, by the compact embedding of $H^1_r(B_1\setminus B_{R_1})$ into $L^p(B_1\setminus B_{R_1})$ for every $p$. Finally, $u_p$ can be chosen nonnegative since $\int_{B_1} |\nabla|u||^2=\int_{B_1}|\nabla u|^2$.
\end{proof}
We are in a position to prove the weak convergence of the minimizers as $p\to\infty$. Let $A=B_{R_2}\setminus B_{R_1}$.
\begin{lemma}\label{lemma:weak_convergence_u_p}
Let $u_{p_n}\in K_{p_n}$ be a sequence of minimizers to $J_{p_n}$. Then there exists $u_\infty \in {H^1_r(B_1)}$ such that, up to a subsequence denoted again by $p_n$, it holds
\begin{equation}\label{a3}
u_{p_n}\rightharpoonup u_\infty \ \text{ in } H^1(B_1), \quad u_{p_n}\to u_\infty \ \text{ in } L^q(B_1), \ \forall q<\infty.
\end{equation}
In addition, $u_{p_n}\to u_\infty$ in $L^\infty(A)$.
\end{lemma}
\begin{proof}
Let us prove that the sequence $J_{p_n}$ is bounded in $H^1(B_1)$, then the statement follows proceeding as in the proof of Proposition \ref{prop:existence_u_p}. To this aim consider a nonnegative test function $\eta \in K_1$ and set
\begin{equation}\label{a5}
\eta_p=\frac{\eta}{\left(|B_1|^{-1}\int_{B_1} \eta^{p+1}\right)^{\frac{1}{p+1}}}.
\end{equation}
Using the H\"older inequality we obtain, for every $p\geq1$
\begin{equation}\label{a6}
Q(\eta_p)=\frac{Q(\eta)}{\left(|B_1|^{-1}\int_{B_1} \eta^{p+1}\right)^{\frac{2}{p+1}}}\le
\frac{Q(\eta)}{|B_1|^{-1}\int_{B_1} \eta^2}=Q(\eta).
\end{equation}
Being $\eta_p\in K_p$, this implies $J_p\leq Q(\eta)$, which concludes the proof.
\end{proof}
The next lemma, roughly speaking, ensures that the ``mass'' of the $u_p$'s concentrates in $A$ as $p\to\infty$.
\begin{lemma}\label{lemma:mass_concentration_in_the_annulus}
It holds
\begin{equation}
\gamma_p=\left(|A|^{-1}\int_A u_p^{p+1}\right)^{\frac{1}{p+1}}\to 1 \quad \text{ as }p\to\infty.
\end{equation}
\end{lemma}
\begin{proof}
Since $u_p\in K_p$, a direct calculation gives
\[
\gamma_p=|A|^{-\frac{1}{p+1}} \left(|B_1|-\int_{B_1\setminus A} u_p^{p+1}\right)^\frac{1}{p+1}.
\]
Now, $\int_{B_1\setminus A} u_p^{p+1}\leq \int_{B_1\setminus A}c^{p+1}\to0$ (being $c<1$), hence $\gamma_p\to1$.
\end{proof}
As a consequence of the previous lemma we deduce that $u_\infty \in K_\infty$, in fact the following holds.
\begin{lemma}\label{lemma:Linfty_norm_of_u_infty}
The limit function $u_\infty$ satisfies $\|u_\infty\|_{L^\infty(A)}=1$.
\end{lemma}
\begin{proof}
Fix $p>1$. Then, the H\"older inequality gives ($\gamma_p$ is defined in Lemma \ref{lemma:mass_concentration_in_the_annulus}),
\begin{equation}\label{a7}
\lim_{q\to\infty} \left(|A|^{-1}\int_A u_p^{q+1}\right)^{\frac{1}{q+1}} \geq \left(|A|^{-1}\int_A u_p^{p+1}\right)^{\frac{1}{p+1}}=\gamma_p.
\end{equation}
By Lemma \ref{lemma:weak_convergence_u_p} we have that $u_p\to u_\infty$ in $L^\infty(A)$  and then
\begin{equation}\label{a8}
\|u_\infty\|_{L^\infty(A)}=\lim_{p\to\infty}\|u_p\|_{L^\infty(A)}=
\lim_{p\to\infty}\lim_{q\to\infty} \left(|A|^{-1}\int_A u_p^{q+1}\right)^{\frac{1}{q+1}}\geq \lim_{p\to\infty}\gamma_p=1.
\end{equation}
Similarly we compute, for a fixed $q>1$,
\begin{equation}\label{a9}
\lim_{p\to\infty} \left(|A|^{-1}\int_A u_p^{q+1}\right)^{\frac{1}{q+1}} \leq \lim_{p\to\infty}
\left(|A|^{-1}\int_A u_p^{p+1}\right)^{\frac{1}{p+1}}=1,
\end{equation}
which gives the opposite inequality,
\begin{equation}\label{a10}
\|u_\infty\|_{L^\infty(A)}=\lim_{q\to\infty}\|u_\infty\|_{L^q(A)}=
\lim_{q\to\infty}\lim_{p\to\infty} \left(|A|^{-1}\int_A u_p^{q+1}\right)^{\frac{1}{q+1}}\leq 1.\qedhere
\end{equation}
\end{proof}
On the other hand we have the following approximation result.
\begin{lemma}\label{lemma:sequence_w_p}
Let $u \in K_\infty$ be a nonnegative function. Then, for any $p>1$, there exists  $w_p\in K_p$, such that $w_p\to u$ in $H^1(B_1)$.
\end{lemma}
\begin{proof}
Let $\vphi(x)=\min(u(x),c)$, defined in $A$. For $\sigma\geq 0$, we define
\begin{equation}\label{a11}
w_\sigma=\left\{ \begin{array}{ll}
u &\text{ in }B_1\setminus A\\
\sigma(u-\vphi)+u&\text{ in }A
\end{array}\right.
\end{equation}
Let us show that, for any $p>1$, there exists $\sigma_p>0$ such that $w_\sigma\in K_p$. To this aim we introduce the function
\begin{equation}\label{a12}
Z(\sigma)=\int_{B_1}|w_\sigma|^{p+1}=|B_1|^{-1}\int_{B_1\setminus A}|u|^{p+1}+|B_1|^{-1}\int_A |\sigma(u-\vphi)+u|^{p+1}.
\end{equation}
$Z$ is continuous, $\lim_{\sigma\to\infty}Z(\sigma)=+\infty$ and moreover
\[
Z(0)=|B_1|^{-1}\int_{B_1}|u|^{p+1}<\|u_p\|_{L^\infty(B_1)}=1,
\]
since $u\not\equiv1$, by the definition of $K_\infty$. Hence there exists $\sigma_p>0$ such that $Z(\sigma_p)=1$. Setting $w_p=w_{\sigma_p}$ and observing that $w_p$ is continuous in $B_1$ and $w_p\in H^1(B_1)$, we conclude that $w_p\in K_p$. Moreover, since $w_p\leq c$ in $B_1\setminus A$, we have
\begin{equation}\label{a13}
\left(|A|^{-1}\int_A |w_p|^{p+1}\right)^{\frac{1}{p+1}}\to1 \quad \text{ as } p\to\infty.
\end{equation}
Let us prove that $\sigma_p\to0$, which concludes the proof. If not, there exists $\delta>0$ such that $\sigma_p>\delta$ for every $p$. This implies $\sigma_p(u-\vphi)\geq \delta (u-\vphi)$ and hence (being $w_p$ nonnegative)
\[
1=\lim_{p\to\infty}\left(|A|^{-1}\int_A |w_p|^{p+1}\right)^{\frac{1}{p+1}} \geq
\lim_{p\to\infty}\left(|A|^{-1}\int_A |(1+\delta)u-\delta\vphi|^{p+1}\right)^{\frac{1}{p+1}}=
\|(1+\delta)u-\delta\vphi\|_{L^\infty(A)}.
\]
Let now $\hat r\in (R_1,R_2)$ be such that $u(\hat r)=1$, then we have
\[
1\geq \|(1+\delta)u-\delta\vphi\|_{L^\infty(A)}\geq (1+\delta)u(\hat r)-\delta \vphi(\hat r)\geq(1+\delta)-\delta c = 1+\delta(1-c)>1,
\]
which is a contradiction.
\end{proof}
The next proposition proves the convergence of the constrained variational problems $J_p$ to the limit problem $J_\infty$ (see \eqref{9}).
\begin{prop}
We have that
\begin{equation}\label{a14}
\lim_{p\to\infty} J_p=J_\infty \quad \text{ and } \quad u_p\to u_\infty \text{ in } H^1(B_1).
\end{equation}
Moreover, $u_\infty\in  K_\infty$ and $Q(u_\infty)=J_\infty$.
\end{prop}
\begin{proof}
We already know that $u_p\rightharpoonup u_\infty$ in $H^1(B_1)$ (Lemma \ref{lemma:weak_convergence_u_p}) and that $u_\infty\in K_\infty$ (Lemma \ref{lemma:Linfty_norm_of_u_infty}). On one hand we have,
\begin{equation}\label{eq:inequalities_convergence_J_p}
J_\infty \leq Q(u_\infty) \leq \liminf_{p\to\infty}Q(u_p)=\liminf_{p\to\infty}J_p,
\end{equation}
where we used the lower semicontinuity of the $H^1$--norm with respect to weak convergence. In order to prove the reverse inequality, let $u \in K_\infty, u\geq0$ and let $w_p$ be the corresponding approximating sequence founded in Lemma \ref{lemma:sequence_w_p}. Then it holds
\begin{equation}\label{a15}
Q(u)=\lim_{p\to\infty}Q(w_p)\geq \limsup_{p\to\infty}J_p.
\end{equation}
Since $J_\infty$ can be equivalently characterized as $\inf\{Q(u):\ u\in K_\infty, \ u\geq0\}$, we have obtained that $J_p\to J_\infty$. As a consequence, the inequalities in \eqref{eq:inequalities_convergence_J_p} are in fact equalities, which implies $Q(u_\infty)=J_\infty$ and also the $H^1$--strong convergence.
\end{proof}

\sezione{Proof Theorem \ref{teo:main_theorem} when $\bar r\in(0,1)$}\label{sec:proof_main_theorem}
The variational characterization of $u_\infty$ proved in the previous section allows to derive the following.
\begin{lemma}\label{lemma:limit_profile}
There exists a unique $r_\infty\in[R_1,R_2]$ such that $u_\infty(r_\infty)=1$. Moreover, $u_\infty(r)$ solves
\begin{equation}\label{b1}
-u''-\frac{n-1}ru'+V(r)u=0, \qquad u'(0)=0
\end{equation}
in $(0,R_1)\cup(R_1,r_\infty)\cup(r_\infty,R_2)\cup(R_2,1)$.
\end{lemma}
\begin{proof}
Let us first show that $u_\infty$ solves the equation in $(0,R_1)$. To this aim let $w$ be the solution of
\begin{equation}\label{b2}
\left\{ \begin{array}{ll}
-w''-\frac{n-1}{r} w'+V(r)w=0 &\text{ in }(0,R_1)\\
w'(0)=0,\ w(R_1)=u_\infty(R_1).
\end{array}\right.
\end{equation}
Hence $w$ minimizes the functional $Q(u)$ in $B_{R_1}$. Moreover the maximum principle ensures
$$0<w<c\quad\hbox{in }B_{R_1}.$$
If we define
\begin{equation}\label{b3}
\widetilde w=\left\{ \begin{array}{ll}
w &\text{ in }B_{R_1}\\
u_\infty&\text{ in }B_1\setminus B_{R_1},
\end{array}\right.
\end{equation}
then $\widetilde w\in K_\infty$ and $Q(\widetilde w)\le Q(u_\infty)$. This implies that $\widetilde w\equiv u_\infty$ in $B_R$ and so $u_\infty$ solves \eqref{b1} in $B_{R_1}$. One can proceed similarly in $B_1\setminus B_{R_2}$.

Set $X=\{r\in[R_1,R_2] : u_\infty(r)=1 \}$ and let $r_m=\inf X, r_M=\sup X$ ($X$ is not empty since $\|u_\infty\|_{L^\infty(A)}=1$). Notice that $u_\infty$ solves \eqref{b1} in the open set $[R_1,R_2]\setminus X$, since here the function does not touch the obstacle. In particular, by the maximum principle, the interval $[r_m,r_M]$ is contained in $X$. Let us end the proof by showing that $X$ is a singleton.
By contradiction assume that , let $w$ be the solution of
\begin{equation}\label{b4}
\left\{ \begin{array}{ll}
-w''-\frac{n-1}rw'+V(r)w=0 &\text{ in }(R_1,r_M)\\
u(R_1)=u_\infty(R_1),\ u(r_M)=1.
\end{array}\right.
\end{equation}
Since $u_\infty\equiv 1$ in $(r_m,r_M)$, whereas $w$ can not be constant in an interval by the strong maximum principle, we have that $Q(w)<Q(u_\infty)$ in the annulus $B_{r_M}\setminus B_{R_1}$, which leads again to a contradiction as before. Thus $X$ is a singleton.
\end{proof}
We deduce that $u_\infty(r)$ is regular in $A$, except for the point $r=r_\infty$, with different right and left derivatives. In the next proposition we show that $r_\infty=\bar r$ (recall that $\bar r$ is a local minimum point of $F$) and that $u_\infty$ coincides in fact with the normalized Green function $G(r,\bar r)$.
\begin{prop}\label{prop:description_of_u_infty}
We have that
\begin{equation}\label{b20}
u_\infty(r)=\frac{G(r,\bar r)}{G(\bar r,\bar r)}.
\end{equation}
\end{prop}
\begin{proof}
Let us first prove that $Q(u_\infty)\geq F(r_\infty)$, with $r_\infty$ given by the previous lemma. To this aim we consider the auxiliary problem
\begin{equation}\label{b5}
\inf\{ Q(u):\ u\in {H^1_r(B_1)}, \ 0\leq u\leq 1 \text{ in } B_1, \ u(r_\infty)=1\}.
\end{equation}
Arguing as in the previous lemma, we have that the function which achieves \eqref{b5} solves the problem
\begin{equation}\label{b6}
\left\{ \begin{array}{ll}
-u''-\frac{n-1}ru'+V(r)u=0 &\text{ in }(0,r_\infty)\cup(r_\infty,1)\\
u'(0)=0,\ u(r_\infty)=1,\ u'(1)=0,\ u\in H^1_r(B_1).
\end{array}\right.
\end{equation}
Since the function $G(r,r_\infty)/G(r_\infty,r_\infty)$ satisfies \eqref{b6}, then we derive that it minimizes \eqref{b5}. Moreover, $u_\infty$ belongs to the minimization set in \eqref{b5}, hence
\begin{equation}
Q(u_\infty)\geq Q\left(\frac{G(\cdot,r_\infty)}{G(r_\infty,r_\infty)}\right)=F(r_\infty)
\end{equation}
(the last equality comes from Lemma \ref{lemma:relation_F_Q}). Note that at this stage we still do dot know whether $G(r,r_\infty)/G(r_\infty,r_\infty)\in K_\infty$.

Next we show that $r_\infty=\bar r$. Assume not. We know that $r_\infty \in [R_1,R_2]$ and that $F(\bar r)<F(r)$ $\forall \ r\in[R_1,R_2]$, since $\bar r$ is strict local minimum point (see Lemma \ref{d1} in Appendix). Hence
\begin{equation}
Q(u_\infty)\geq F(r_\infty)>F(\bar r)=Q\left(\frac{G(\cdot,\bar r)}{G(\bar r,\bar r)}\right).
\end{equation}
Now, due to our choice of the constant $c$, we have that $G(r,\bar r)/G(\bar r,\bar r) \in K_\infty$, hence the last inequality gives a contradiction. We conclude that $r_\infty=\bar r$ and, in turn, that $u_\infty=G(r,\bar r)/G(\bar r,\bar r)$.
\end{proof}
\begin{proof}[Proof Theorem \ref{teo:main_theorem} when $\bar r\in(0,1)$]
Let us consider the function $u_p$ which minimizes $J_p$. Then, proceeding exactly as in Lemma 2.6 in \cite{P}, we have that
\begin{equation}\label{b21}
-\Delta u_p+V(|x|)u_p\le \lambda_pu_p^p\quad\hbox{in }B_1,
\end{equation}
and
\begin{equation}\label{b22}
-\Delta u_p+V(|x|)u_p=\lambda_pu_p^p\quad\hbox{in }A,
\end{equation}
where $\lambda_p$ is a Lagrange multiplier. We want to show that $|\lambda_p|\le C$ where $C$ is a positive constant independent of $p$.
First note that multiplying \eqref{b21} by $u_p$ and integrating in $B_1$ we immediately obtain that $\lambda_p$ is positive.  In order to get a bound from above to $\lambda_p$ let us choose $R_1< \overline R_1<r_\infty <\overline R_2 < R_2$ such that $\int_{\partial B_{\overline R_1}}|\nabla u_p|^2$ and $\int_{\partial B_{\overline R_2}}|\nabla u_p|^2$ are uniformly bounded in $p$ (this is possible since $\int_{B_1}|\nabla u_p|^2$  is uniformly bounded). Then multiply \eqref{b22} by $u_p$ and integrate in $\bar A=B_{\overline R_2} \setminus B_{\overline R_1}$. We get
\begin{equation}\label{b22bis}
\lambda_p\int_{\bar A}|u_p|^{p+1}\le J_p+\int_{\partial \bar A}\frac{\partial u_p}{\partial\nu}u_p
\le J_p+\left(\int_{\partial \bar A}|\nabla u_p|^2\right)^\frac12
\left(\int_{\partial \bar A}|u_p|^2\right)^\frac12.
\end{equation}
Now, proceeding as in Lemma \ref{lemma:mass_concentration_in_the_annulus}, one can show that $\int_{\bar A} u_p^{p+1}\rightarrow1$ as $p\to\infty$ (since $\overline R_1 <r_\infty<\overline R_2$). Hence \eqref{b22bis} implies that $\lambda_p\le C$, for some constant $C$ independent of $p$.

Finally we claim that
\begin{equation}\label{b23}
u_p\rightarrow u_\infty\quad\hbox{as }p\rightarrow\infty\hbox{ in }L^\infty(B_1).
\end{equation}
Notice first that $u_p\to u_\infty$ in $L^\infty(B_1\setminus B_\epsilon)$ for any $\epsilon>0$, because of the embedding of $H^1_r(B_1\setminus B_\epsilon)$ into $L^\infty(B_1\setminus B_\epsilon)$. Recalling the equation satisfied by $u_\infty$ in $B_{R_1}$, we get from \eqref{b21} that
\begin{equation}\label{b24}
-\Delta\left(u_p-u_\infty\right)+V(x)\left(u_p-u_\infty\right)\le \lambda_pu_p^p\quad\hbox{in }B_{R_1}.
\end{equation}
By known regularity results (see for example Theorem 9.1 in \cite{GT}), we have that
\begin{equation}\label{b25}
||u_p-u_\infty||_{L^\infty(B_{R_1})}\le ||u_p-u_\infty||_{L^\infty(\partial B_{R_1})}+C||\lambda_pu_p^p||_{L^n(B_{R_1})}.
\end{equation}
Being $u_p<1$ in $B_{R_1}$ and $\{\lambda_p\}$ a bounded sequence, we deduce that the right hand side in the previous inequality converges to zero as $p\to\infty$, hence \eqref{b23} is proved.

Now, as a consequence of Proposition \ref{prop:description_of_u_infty}, of the uniform convergence and of the choice of $c$ in \eqref{eq:property_c}, we deduce that
\begin{equation}\label{b27}
u_p<c\quad\hbox{in }B_{R_1}\cup (B_1\setminus B_{R_2}).
\end{equation}
Hence $u_p$ solves
\begin{equation}\label{b28}
-\Delta u_p+V(|x|) u_p=\lambda_p u_p^p \quad \text{in } B_1,
\end{equation}
for a Lagrange multiplier $\lambda_p>0$. A suitable multiple of $u_p$ provides a solution to \eqref{eq:main_nonlinear_equation_neumann}.
\end{proof}

Being $\bar r$ a critical point of $F(r)$, we have that the function $G(r,\bar r)$ verifies the following interesting $reflection\ principle$.
\begin{prop}\label{prop:reflection_principle}
We have that
\begin{equation}\label{c1}
\lim\limits_{r\rightarrow\bar r^-}G_r(r,\bar r)=
-\lim\limits_{r\rightarrow\bar r^+}G_r(r,\bar r)=\frac{1}{2},
\end{equation}
where $G_r(r,s)$ denotes the derivative with respect to $r$.
\end{prop}
\begin{proof}
Lemma \ref{lemma:appendix_factorization_of_G} gives $F(r)=|\partial B_1|/[\xi(r)\zeta(r)]$. Hence, being $\bar r$ a critical point of $F(r)$, we have
\begin{equation}\label{c5}
\xi'(\bar r)\zeta(\bar r)+\xi(\bar r)\zeta'(\bar r)=0.
\end{equation}
On the other hand, again by Lemma \ref{lemma:appendix_factorization_of_G} we have that
\begin{equation}\label{c6}
\lim\limits_{r\rightarrow\bar r^-}G_r(r,\bar r)=\bar r^{n-1}\xi'(\bar r)\zeta(\bar r), \qquad
\lim\limits_{r\rightarrow\bar r^+}G_r(r,\bar r)=\bar r^{n-1}\xi(\bar r)\zeta'(\bar r),
\end{equation}
which, together with \eqref{c5}, gives the first equality. In order to obtain the value of the left derivative it is enough to combine \eqref{c5} with \eqref{eq:xi'zeta-xizeta'} at $\bar r$.
\end{proof}

\sezione{The case $\bar r=1$}\label{sec:proof_teo_Neumann_bar_r=1}
In this section we conclude the proof of Theorem \ref{teo:main_theorem}, dealing with the case $\bar r=1$, and we prove Theorem \ref{teo:Neumann_bar_r=1}. Let us start by showing that Lemma \ref{lemma:relation_F_Q} still holds at $r=1$. Recall that $G(r,1)$ is well defined (Lemma \ref{lemma:appendix_factorization_of_G} and Lemma \ref{z}) and it is the punctual limit of $G(r,s)$ as $s\to1$.
\begin{lemma}\label{lemma:relation_F_Q_generalization}
We have that
\begin{equation}
F(1)=Q\left(\frac{G(\cdot,1)}{G(1,1)}\right).
\end{equation}
\end{lemma}
\begin{proof}
It comes from Lemma \ref{lemma:appendix_factorization_of_G} that $G(r,1)/G(1,1)$ is well defined (since $\zeta(1)\neq0$) and moreover $G(r,1)/G(1,1)=\xi(r)/\xi(1)$. In order to evaluate the energy of this function, let us write down the equation satisfied by $\xi$ and multiply by $r^{n-1}\xi$. We have,
\begin{equation}
V(r)\xi^2 r^{n-1}=r^{n-1}\xi''\xi+(n-1) r^{n-2}\xi'\xi =\frac{d}{dr}\left( r^{n-1}\xi'\xi \right) - r^{n-1}(\xi')^2,
\end{equation}
and hence
\begin{equation}
\int_0^1 \left[(\xi')^2 + V(r) \xi^2 \right]r^{n-1} dr =\xi'(1) \xi(1).
\end{equation}
As a consequence it holds
\begin{equation}
Q\left(\frac{G(\cdot,1)}{G(1,1)}\right)=\frac{|\partial B_1| \xi'(1)}{\xi(1)}=\frac{|\partial B_1| \xi'(1)\zeta(1)}{\xi(1)\zeta(1)}=
\frac{|\partial B_1|}{\xi(1)\zeta(1)}=F(1),
\end{equation}
where we used \eqref{eq:xi'zeta-xizeta'} at $r=1$ and that $\zeta'(1)=0$.
\end{proof}
\begin{proof}[Proof of Theorem \ref{teo:main_theorem} when $\bar r=1$]
This is analogous to the case $\bar r\in(0,1)$. Choose $R_1$ and $c$ in such a way that $1$ is a global minimum point in $[R_1, 1]$ and
\begin{equation}
\frac{G(R_1,1)}{G(1,1)}<c<1.
\end{equation}
In analogy with Section \ref{s0} we set
\begin{equation*}
K_p=\left\{u\in H^1_r(B_1):\ \left(|B_1|^{-1}\int_{B_1} |u|^{p+1}\right)^{\frac{1}{p+1}}=1 \ \hbox{ and }\ |u|\leq c \text{ in } B_{R_1} \right\},
\end{equation*}
and we define correspondingly $J_p$ and $J_\infty$. Thanks to Lemma \ref{lemma:relation_F_Q_generalization}, it is possible to proceed as in Sections \ref{sec:existence_convergence_minimizers}, \ref{sec:proof_main_theorem} with minor changes. This provides the existence of a nonconstant solution of the Neumann problem for $p$ sufficiently large.
\end{proof}

\begin{proof}[Proof of Theorem \ref{teo:Neumann_bar_r=1}] It is sufficient so show that $\bar r=1$ is a local minimum point of $F(r)$, for every choice of $V(|x|)\geq0$, $V\not\equiv0$. To this aim we compute the derivatice of $F$ as follows
\begin{equation}
F'(r)=-|\partial B_1|\cdot \frac{\xi'(r)\zeta(r)+\xi(r)\zeta'(r)}{\left(\xi(r)\zeta(r)\right)^2}=
-|\partial B_1|\cdot \frac{\xi'(r)\zeta(r)\cdot 1+\xi(r)\zeta'(r)\cdot 1}{\left(\xi(r)\zeta(r)\right)^2}.
\end{equation}
Using \eqref{eq:xi'zeta-xizeta'} we have that $1=r^{n-1}\xi'(r)\zeta(r)-r^{n-1}\xi(r)\zeta'(r)$, which, substituted in the previous equality, gives
\begin{equation}
F'(r)=r^{n-1} |\partial B_1| \cdot \frac{\left(\xi(r)\zeta'(r)\right)^2-\left(\xi'(r)\zeta(r)\right)^2}{\left(\xi(r)\zeta(r)\right)^2}.
\end{equation}
Now, the boundary condition gives $\zeta'(1)=0$, whereas \eqref{eq:xi'zeta-xizeta'} at $r=1$ implies $\xi'(1)\zeta(1)\neq0$, therefore
\begin{equation}
\lim_{r\to 1^-} F'(r)=-|\partial B_1| \cdot \frac{\left(\xi'(1)\zeta(1)\right)^2}{\left(\xi(1)\zeta(1)\right)^2}<0
\end{equation}
which concludes the proof.
\end{proof}

\sezione{Appendix}
We collect some properties of the Green function $G(r,s)$.
\begin{lemma}\label{z}
Let $V(|x|)\geq0$, $V\not\equiv0$ be a smooth radial function in $B_1$. There exist linearly independent solutions $\xi,\zeta \in C^2((0,1])$ of the equation
\begin{equation}\label{eq:homogeneous_equation_appendix}
-u''(r)-\frac{n-1}{r}u'(r) +V(r)u(r)=0, \quad u>0,
\end{equation}
satisfying $\xi'(0)=\zeta'(1)=0$ and enjoying the additional property
\begin{equation}\label{eq:xi'zeta-xizeta'}
\xi'(r)\zeta(r)-\xi(r)\zeta'(r)\equiv \frac{1}{r^{n-1}}, \qquad r\in (0,1].
\end{equation}
The same result holds in case of Dirichlet boundary conditions at $r=1$, that is $\xi'(0)=\zeta(1)=0$.
\end{lemma}
\begin{proof}
In case of Dirichlet boundary conditions the result is proved by Catrina in \cite{C}, Appendix. Let us adapt the proof to the case of Neumann
boundary conditions. Let $s=r^{2-n}$ and $u(r)=\tilde u(r^{2-n})$, $V(r)=\tilde V(r^{2-n})$, then \eqref{eq:homogeneous_equation_appendix} transforms into
\begin{equation}\label{eq:transformed_eq_appendix}
-\tilde u''(s)+\frac{s^\frac{2n-2}{2-n}}{(n-2)^2} \tilde V(s)\tilde u(s)=0, \qquad s\in [1,\infty).
\end{equation}
Catrina provides, via an approximation method, a positive function $\vphi(s)$ which satisfies
\begin{equation}
\vphi(s)=1+\int_s^\infty\left( 1-\frac{s}{t} \right) \frac{\tilde V(t)}{(n-2)^2}\vphi(t)\ t^\frac{n}{2-n} \ dt,
\end{equation}
and hence solves \eqref{eq:transformed_eq_appendix}. The function $\xi(r)=\vphi(r^{2-n})/(n-2)$ solves \eqref{eq:homogeneous_equation_appendix} and moreover, as shown in \cite{C}, $\lim_{r\to 0^+} \xi'(r)=0$. Next we set
\begin{equation}
\psi(s)=\vphi(s)\left\{-\frac{1}{\vphi(1)\vphi'(1)}+\int_1^s\frac{1}{\vphi^2(t)} \ dt\right\},
\end{equation}
which is well defined since
\begin{equation}
\vphi'(1)=\vphi(1)-1=\int_1^\infty\left(1-\frac{1}{t}\right)\frac{\tilde V(t)}{(n-2)^2}\vphi(t) \ t^\frac{n}{2-n} \ dt\ >0,
\end{equation}
by the assumption $V\geq0$, $V\not\equiv0$. A direct calculation shows that
\begin{equation}
\psi'(1)=0 \quad\text{ and }\quad \vphi(s)\psi'(s)-\vphi'(s)\psi(s)\equiv 1 \quad s\in[1,\infty),
\end{equation}
hence the pair $\zeta(r)=\psi(r^{2-n})$, $\xi(r)$ satisfies \eqref{eq:xi'zeta-xizeta'} and the lemma is proved.
\end{proof}
A straightforward consequence of this result is the following factorization of the Green function.
\begin{lemma}\label{lemma:appendix_factorization_of_G}
In the assumptions of the previous lemma it holds
\begin{equation}
G(r,s)=\left\{ \begin{array}{ll}
         s^{n-1}\xi(r)\zeta(s) \quad &\text{for } \ r\leq s\\
	 s^{n-1}\xi(s)\zeta(r) \quad &\text{for } \ r> s,
        \end{array}
\right.
\end{equation}
accordingly to the choice of the boundary conditions.
\end{lemma}
\begin{lemma}\label{d1}
Let $\bar r$ be a local minimum point of the function $F(r)$. Then $\bar r$ is a strict local minimum point.
\end{lemma}
\begin{proof}
The previous lemma gives $F(r)=|\partial B_1|/[\xi(r)\zeta(r)]$, which is constant in an interval $(a,b)$ if and only if
\begin{equation}
\xi(r)=\frac{C}{\zeta(r)} \quad \text{ in } (a,b)
\end{equation}
for some positive constant $C$ (recall that $\xi$ and $\zeta$ are strictly positive). But this can not happen because of \eqref{eq:xi'zeta-xizeta'}.
\end{proof}

\subsection*{Acknowledgements}
The authors wish to thank Susanna Terracini for the fruitful discussion.
Work was partially supported by MIUR, Project “Metodi Variazionali ed Equazioni Differenziali Non Lineari.”


\begin{thebibliography}{1}

\bibitem{AM}
{\sc Adimurthi and G. Mancini,}
\newblock The Neumann problem for elliptic equations with critical nonlinearity,
\newblock {\em Nonlinear analysis, Sc. Norm. Super. di Pisa },  (1991), 9--25.

\bibitem{APY}
{\sc Adimurthi, F. Pacella and S.L. Yadava,}
\newblock Interaction between the geometry of the boundary and positive solutions of a semilinear Neumann problem with critical nonlinearity,
\newblock {\em J. Funct. Anal.  113 },  (1993), 318--350.

\bibitem{AY}
{\sc Adimurthi and S.L. Yadava,}
\newblock Critical Sobolev exponent problem in $\R^n (n\ge4)$ with Neumann boundary condition,
\newblock {\em Proc. Indian Acad. Sci. Math. Sci.  100}, (1990), 275--284.

\bibitem{AY91}
{\sc Adimurthi and S.~L. Yadava,}
\newblock Existence and nonexistence of positive radial solutions of {N}eumann
  problems with critical {S}obolev exponents.
\newblock {\em Arch. Rational Mech. Anal.}, 115(3):275--296, 1991.

\bibitem{AY93}
{\sc Adimurthi and S.~L. Yadava,}
\newblock On a conjecture of {L}in-{N}i for a semilinear {N}eumann problem.
\newblock {\em Trans. Amer. Math. Soc.}, 336(2):631--637, 1993.

\bibitem{BC}
{\sc A. Bahri and J. M. Coron,}
\newblock On a nonlinear elliptic equation involving the critical Sobolev exponent: the effect of the topology of the domain,
\newblock {\em Comm. Pure Appl. Math.  41}, (1988), 253-294.

\bibitem{BN}
{\sc H. Brezis and L. Nirenberg,}
\newblock Positive solutions of nonlinear elliptic equations involving critical Sobolev exponents,
\newblock {\em Comm. Pure Appl. Math.  36}, (1983), 437-477.

\bibitem{C}
{\sc F. Catrina,}
\newblock A note on a result of {M}. {G}rossi,
\newblock {\em Proc. Amer. Math. Soc. 137}, 11 (2009), 3717--3724.

\bibitem{CK91}
{\sc M.~Comte and M.~C. Knaap,}
\newblock Existence of solutions of elliptic equations involving critical
  {S}obolev exponents with {N}eumann boundary condition in general domains.
\newblock {\em Differential Integral Equations}, 4(6):1133--1146, 1991.

\bibitem{DW}
{\sc M. del Pino and J. Wei,}
\newblock Supercritical elliptic problems in domains with small holes,
\newblock {\em Ann. Inst. H. Poincar\'e Anal. Non Lin\'eaire  24}, 24  (2007), 507-520.

\bibitem{GNN}
{\sc B. Gidas, W. M. Ni and L. Nirenberg,}
\newblock Symmetry of positive solutions of nonlinear elliptic equations in $R^{n}$.
\newblock Mathematical analysis and applications, Part A, Adv. in Math. Suppl. Stud., 7a, Academic Press, New York-London, (1981), 369--402.

\bibitem{GT}
{\sc D. Gilbarg and N. Trudinger,}
\newblock Elliptic Partial Differential Equations of Second Order. Second edition.
\newblock Springer-Verlag, Berlin, 1983.

\bibitem{G1}
{\sc M. Grossi,}
\newblock Asymptotic behaviour of the {K}azdan-{W}arner solution in the
  annulus,
\newblock {\em J. Differential Equations 223}, 1 (2006), 96--111.

\bibitem{G2}
{\sc M. Grossi,}
\newblock  Radial solutions for the Brezis-Nirenberg problem involving large nonlinearities,
\newblock {\em  Jour. Funct. Anal. 250}, (2008), 2995-3036.

\bibitem{LN}
{\sc Chang~Shou Lin and Wei-Ming Ni,}
\newblock On the diffusion coefficient of a semilinear {N}eumann problem.
\newblock In {\em Calculus of variations and partial differential equations
  ({T}rento, 1986)}, volume 1340 of {\em Lecture Notes in Math.}, pages
  160--174. Springer, Berlin, 1988.


\bibitem{MP}
{\sc F. Merle and L. A. Peletier,}
\newblock Positive solutions of elliptic equations involving supercritical growth,
\newblock {\em  Proc. Roy. Soc. Edinburgh Sect. A  118}, (1991), 49-62.

\bibitem{MPS}
{\sc F. Merle, L. A. Peletier and J. Serrin,}
\newblock  A bifurcation problem at a singular limit,
\newblock {\em  Indiana Univ. Math. J.  43}, (1994),585-609.

\bibitem{N}
{\sc W. M. Ni,}
\newblock  On the positive radial solutions of some semilinear elliptic equations on $R^{n}$,
\newblock {\em  Appl. Math. Optim.  9}, (1983), 373-380.

\bibitem{NPT}
{\sc W. M. Ni, X. B. Pan and I. Takagi,}
\newblock  Singular behavior of least-energy solutions of a semilinear Neumann problem involving
critical Sobolev exponents,
\newblock {\em  Duke Math. J. 67}, (1992), 1--20.




\bibitem{P}
{\sc D. Passaseo,}
\newblock Existence and multiplicity of positive solutions for elliptic
  equations with supercritical nonlinearity in contractible domains,
\newblock {\em Rend. Accad. Naz. Sci. XL Mem. Mat. 16} (1992), 77-98.

\bibitem{P1}
{\sc D. Passaseo,}
\newblock Nonexistence results for elliptic problems with supercritical nonlinearity in nontrivial domains,
\newblock {\em J. Funct. Anal.  114} (1993), 97-105.

\end{thebibliography}
\end{document}